\documentclass[leqno,12pt]{article} 
\usepackage{amsmath, amssymb}
\usepackage{amsthm} 
\usepackage{epsfig}

\theoremstyle{plain} 
\newtheorem{theorem}{\indent\sc Theorem}[section]

\newtheorem{proposition}[theorem]{\indent\sc Proposition}

\theoremstyle{definition} 

\newtheorem{remark}[theorem]{\indent\sc Remark}
\newtheorem{example}[theorem]{\indent\sc Example}

\newcommand\on{\operatorname}
\renewcommand\div{\on{div}}

\newcommand\Ric{\on{Ric}}

\newcommand\trace{\on{trace}}

\title{Some properties of hyperbolic Yamabe solitons}

\author{Adara M. Blaga and Cihan \"{O}zg\"{u}r}

\date{}

\begin{document}

\maketitle

\markboth{{\small\it {\hspace{2cm} Some properties of hyperbolic Yamabe solitons}}}{\small\it{Some properties of hyperbolic Yamabe solitons \hspace{2cm}}}

\footnote{ 
2020 \textit{Mathematics Subject Classification}.
37K40; 53C21; 53C25; 53C50; 53Z05.
}
\footnote{ 
\textit{Key words and phrases}.
Hyperbolic Yamabe flow, hyperbolic Yamabe soliton, concurrent vector field, immersed submanifold, warped product manifold.
}

\begin{abstract}
We define the hyperbolic Yamabe flow and obtain some properties of its stationary solutions, namely, of hyperbolic Yamabe solitons. We consider immersed submanifolds as hyperbolic Yamabe solitons and prove that, under certain assumptions, a hyperbolic Yamabe soliton hypersurface is a pseudosymmetric or a metallic shaped hypersurface. We characterize the hyperbolic Yamabe soliton factor manifolds of a multiply twisted, multiply warped, doubly warped, and warped product manifold and provide a classification for a complete gradient hyperbolic Yamabe soliton factor manifold. We also determine the conditions for the factor manifolds to be hyperbolic Yamabe solitons if the manifold is a hyperbolic Yamabe soliton and illustrate this result for a physical model of the universe, namely, for the Robertson--Walker spacetime.
\end{abstract}

\section{Introduction}

After the fundamental work of Hamilton \cite{ham}, where he introduced the Ricci flow in order to study compact $3$-dimensional manifolds with positive Ricci curvature, the theory of solitons, as stationary solutions of various geometric flows, has been lately developed a lot. New geometrical properties of solitons have been determined, especially connected to curvature, and obstructions for a manifold to be a soliton have been established.
Yamabe flow was introduced also by Hamilton in order to solve the Yamabe problem \cite{hamil}. While it is known that the Ricci flow controls the heat diffusion, in 2006, Kong and Liu considered the hyperbolic Ricci flow (called in that paper, the hyperbolic geometric flow), a useful tool in modeling and understanding the wave phenomena and also the geometry of a Riemannian manifold (see \cite{li}). For some recent developments into the theory of hyperbolic Ricci solitons, we can refer to \cite{az1,az2,ad1,ac9,ac4,az3,kaci, paha,sid, sidi, yo}.


Motivated by these studies, in the present paper, we consider a new evolution equation, namely, the hyperbolic Yamabe flow. In Section 2, we focus on studying some properties of hyperbolic Yamabe solitons, providing also conditions for its potential vector field to be Killing. In Section 3, we consider isometrically immersed submanifolds as hyperbolic Yamabe solitons having as potential vector field the tangential component $\zeta^{\top}$ of a concurrent vector field $\zeta$ on the ambient manifold and obtain some properties of the shape operator. In particular, we deduce that the totally geodesic submanifolds are hyperbolic Yamabe solitons with potential vector field $\zeta^{\top}$ if and only if they have constant scalar curvature, and, if the submanifold is a hyperbolic Yamabe soliton hypersurface with parallel shape operator isometrically immersed into a space of constant sectional curvature, then it is pseudosymmetric and metallic shaped. In Section 4, we focuse on manifolds of warped product-type.
We determine the necessary and sufficient conditions for the factor manifolds of a multiply twisted, multiply warped, doubly warped, and warped product manifold to be hyperbolic Yamabe solitons. We show that, if any of the factor manifolds $(M_i,g_i,\zeta_i)$ of a multiply twisted, multiply warped, doubly warped, or warped product manifold endowed with a $2$-Killing vector field $\zeta$ is a complete hyperbolic Yamabe soliton having as potential vector field of gradient type, then it is isometric to a complete warped product manifold. Finally, we determine the conditions when the factor manifolds are hyperbolic Yamabe solitons if the manifold is a hyperbolic Yamabe soliton and we illustrate this result for a particular case of warped product manifold, namely, for the Robertson--Walker spacetime. More about the physical meaning of spacetimes can be found, for example, in \cite{car}.
Also, it is worth mentioning that the differentiable manifolds of warped product-type are suitable modeling spaces for various physical phenomena, playing an essential role especially in the theory of relativity. We refer to the following papers for the definitions of a warped product manifold \cite{Bi}, a doubly warped product \cite{erlich}, a multiply warped product \cite{unal}, and
a multiply twisted product \cite{wang}. More about their applications can be found in \cite{On}.

\section{Hyperbolic Yamabe flow and hyperbolic Yamabe solitons}

We introduce here the notion of \textit{hyperbolic Yamabe flow} as being an evolution equation
\begin{align}\label{fl}
\frac{\partial^2 g}{\partial t^2}(t)&=-r(t)g(t),
\end{align}
for $g$ a time-depending Riemannian metric on a smooth manifold $M$ (where $r$ stands for the scalar curvature).
We consider the family of vector fields $\zeta(t):=\displaystyle \frac{1}{f(t)}\zeta _0$ tangent to $M$, for $f$ a smooth positive real function from a real open interval $(a,b)$ containing $0$ with $f(0)=1$ (and $\zeta_0=\zeta(0)$), and we denote by $\{\varphi_t\}_{t\in (a,b)}$ the group of diffeomorphisms of $\{\zeta(t)\}_{t\in (a,b)}$. We define the family of Riemannian metrics $\{g(t):=f(t)\varphi_t^*(g_0)\}_{t\in (a,b)}$ (with $g_0=g(0)$). Then:
\begin{align*}
\frac{\partial g}{\partial t}(t)&=\frac{\partial }{\partial t}\left(f(t)\varphi_t^*(g_0)\right)\\
&=f'(t)\varphi_t^*(g_0)+\varphi_t^*(\pounds_{\zeta_0}g_0),\\
\frac{\partial^2 g}{\partial t^2}(t)&=\frac{\partial }{\partial t}\left(f'(t)\varphi_t^*(g_0)+\varphi_t^*(\pounds_{\zeta_0}g_0)\right)\\
&=f''(t)\varphi_t^*(g_0)+\displaystyle \frac{f'(t)}{f(t)}\varphi_t^*(\pounds_{\zeta_0}g_0)+\displaystyle \frac{1}{f(t)}\varphi_t^*(\pounds_{\zeta_0}\pounds_{\zeta_0}g_0),
\end{align*}
where $\pounds_{\zeta_0}g_0$ is the Lie derivative of the metric $g_0$ in the direction of $\zeta_0$ and $\pounds_{\zeta_0}\pounds_{\zeta_0}g_0:=\pounds_{\zeta_0}(\pounds_{\zeta_0}g_0)$ is the second Lie derivative of $g_0$ in the direction of $\zeta_0$.
By means of \eqref{fl}, we infer
\begin{align*}
\pounds_{\zeta_0}\pounds_{\zeta_0}g_0+f'(0)\pounds_{\zeta_0}g_0=\left(-f''(0)-r_0\right)g_0,
\end{align*}
where $r_0=r(0)$ is the scalar curvature of $(M,g_0)$. Now, choosing $f$ to satisfy $f'(0)=\lambda\in \mathbb R$ and $f''(0)=-\mu\in \mathbb R$, we get the equation of the stationary solutions of the hyperbolic Yamabe flow, namely, the \textit{hyperbolic Yamabe solitons}:
\begin{align}\label{e1}
\pounds_{\zeta}\pounds_{\zeta}g+\lambda \pounds_{\zeta}g&=(\mu-r)g
\end{align}
for $\zeta$ a vector field tangent to $M$ and $\lambda,\mu\in \mathbb R$. In the sequel, we will use the notation $(M^n,g,\zeta,\lambda,\mu)$ for an $n$-dimensional hyperbolic Yamabe soliton and we will call $\zeta$ the potential vector field.

Since we know from \cite{a, ac0} that
\begin{align}
\div(\pounds_{\zeta}g)&=2\left(d(\div(\zeta))+i_{Q\zeta}g\right),\label{e2}\\
\trace(\pounds_{\zeta}\pounds_{\zeta}g)&=2\left(\Vert \nabla {\zeta}\Vert^2+\div(\nabla_{\zeta}\zeta)-\Ric(\zeta,\zeta)\right),\label{e3}
\end{align}
where $\Ric$ is the Ricci curvature tensor field, $Q$ is the Ricci operator defined by $g(QX,Y):=\Ric(X,Y)$ for any vector fields $X,Y$ tangent to $M$, and $\nabla$ is the Levi-Civita connection of $g$, we can firstly conclude

\begin{proposition}
Let $(M^n,g,\zeta,\lambda,\mu)$ be a hyperbolic Yamabe soliton.

(i) If $M$ is closed and $\pounds_{\zeta}\pounds_{\zeta}g$ is trace-free, then $$\int_M(\mu-r)=0.$$

Moreover, if $\int_M\Ric(\zeta,\zeta) \leq 0$, then $\zeta$ is a parallel vector field (i.e., $\nabla \zeta=\nolinebreak 0$).

(ii) If $\pounds_{\zeta}\pounds_{\zeta}g$ is divergence-free, then
$$\nabla r=-2\lambda\left(\nabla (\div(\zeta))+Q\zeta\right).$$

Moreover, if $r$ is constant, then either $\lambda=0$ or $Q\zeta=-\nabla (\div(\zeta))$ (is a gradient vector field).
In the second case, if $M$ is compact, connected, and $\lambda\neq 0$, then $\Ric(\zeta,\zeta)=0$.
\end{proposition}
\begin{proof}
(i) If $\trace(\pounds_{\zeta}\pounds_{\zeta}g)=0$, by taking the trace into the soliton equation \eqref{e1}, we have
$$2\lambda \div(\zeta)=(\mu-r)n,$$
which, by integration, gives $\int_M(\mu-r)=0$ by means of the divergence theorem. Also, from \eqref{e3}, we have
$$\Vert \nabla {\zeta}\Vert^2=\Ric(\zeta,\zeta)-\div(\nabla_{\zeta}\zeta),$$
hence
$$\int_M\Vert \nabla {\zeta}\Vert^2=\int_M\Ric(\zeta,\zeta)\leq 0,$$
and we get $\nabla \zeta=0$.

(ii) If $\div(\pounds_{\zeta}\pounds_{\zeta}g)=0$, and taking into account that we have (see \cite{a})
$$(\div(\pounds_{\zeta}g))(X)=2X(\div(\zeta))+2\Ric(X,\zeta)$$
for any vector field $X$ tangent to $M$, by taking the divergence into the soliton equation \eqref{e1}, we get
$$\lambda \div(\pounds_{\zeta}g)=-dr,$$
and, from \eqref{e2}, we obtain
$$g(\nabla r,X)=-2\lambda\left(g(\nabla(\div(\zeta)),X)+g(Q\zeta,X)\right)$$
for any vector field $X$ tangent to $M$, and we get
\begin{align}\label{r1}
\nabla r&=-2\lambda\left(\nabla (\div(\zeta))+Q\zeta\right).
\end{align}
In particular, if $r$ is constant, then
$$\lambda (\nabla(\div(\zeta))+Q\zeta)=0,$$
which implies $\lambda=0$ or $Q\zeta=-\nabla (\div(\zeta))$. In the second case,
$$\Ric(\zeta,\zeta)=-\zeta(\div(\zeta)).$$
But since $\div(\pounds_{\zeta}\pounds_{\zeta}g)=0$ and $M$ is connected, $\trace(\pounds_{\zeta}\pounds_{\zeta}g)$ is a constant. By taking the trace in \eqref{e1} and then differentiating it in the direction of $\zeta$, we get
$$\zeta(\trace(\pounds_{\zeta}\pounds_{\zeta}g))+2\lambda \zeta(\div(\zeta))=-n\zeta(r),$$
which gives $\zeta(\div(\zeta))=0$, therefore, $\Ric(\zeta,\zeta)=0$, and the proof is complete.
\end{proof}

\begin{remark}
If $(M,g,\zeta,\lambda,\mu)$ is a hyperbolic Yamabe soliton and if $\div(\pounds_{\zeta}\pounds_{\zeta}g)=0$, then, by taking the divergence in \eqref{r1},
we get
$$\Delta( r)=-2\lambda\left(\Delta (\div(\zeta))+\div(Q\zeta)\right),$$
therefore, the scalar curvature is a harmonic function if and only if $\lambda=0$ or if $\zeta$ satisfies $\Delta (\div(\zeta))=-\div(Q\zeta)$.
\end{remark}

\begin{proposition}
Let $(M^n,g,\zeta,\lambda,\mu)$ be a compact hyperbolic Yamabe soliton with divergence-free potential vector field $\zeta$ and $\lambda\neq 0$. Then
$$\int_M\Vert\pounds_{\zeta}\pounds_{\zeta}g\Vert^2=2\lambda^2\int_M\left(\Vert\nabla\zeta\Vert^2-\Ric(\zeta,\zeta)\right)+n\int_M(\mu-r)^2.$$

Moreover, if one of the following conditions holds:

(i) $\int_M\Vert\pounds_{\zeta}\pounds_{\zeta}g\Vert^2\leq n\int_M(\mu-r)^2$;

(ii) $2\lambda^2\int_M\left(\Ric(\zeta,\zeta)-\Vert\nabla\zeta\Vert^2\right)\geq n\int_M(\mu-r)^2$,\\ then $\zeta$ is a Killing vector field (i.e., $\pounds_{\zeta}g=0$) and the scalar curvature $r$ is constant.
\end{proposition}
\begin{proof}
Computing the Hilbert--Schmidt norms into the soliton equation \eqref{e1}, we infer
\begin{align*}
\Vert\pounds_{\zeta}\pounds_{\zeta}g\Vert^2&=(\mu-r)^2\Vert g\Vert^2-2(\mu-r)\lambda\langle g,\pounds_{\zeta}g\rangle+\lambda^2\Vert\pounds_{\zeta}g\Vert^2\\
&=n(\mu-r)^2-4(\mu-r)\lambda\div(\zeta)+\lambda^2\Vert\pounds_{\zeta}g\Vert^2\\
&=n(\mu-r)^2+\lambda^2\Vert\pounds_{\zeta}g\Vert^2.
\end{align*}

Since we have \cite{ya}
\begin{align*}
0&=\int_M\left(\Ric(\zeta,\zeta)+\frac{1}{2}\Vert \pounds_{\zeta}g\Vert^2-\Vert \nabla\zeta\Vert^2-(\div(\zeta))^2\right)\\
&=\int_M\left(\Ric(\zeta,\zeta)+\frac{1}{2}\Vert \pounds_{\zeta}g\Vert^2-\Vert \nabla\zeta\Vert^2\right),
\end{align*}
we get
\begin{align*}
\int_M\left(\Vert\pounds_{\zeta}\pounds_{\zeta}g\Vert^2-n(\mu-r)^2\right)&=\lambda^2\int_M\Vert\pounds_{\zeta}g\Vert^2\\
&=2\lambda^2\int_M\left(\Vert \nabla\zeta\Vert^2-\Ric(\zeta,\zeta)\right),
\end{align*}
and we get the required relation.

If $\int_M\Vert\pounds_{\zeta}\pounds_{\zeta}g\Vert^2\leq n\int_M(\mu-r)^2$ and $\lambda \neq 0$, then $\pounds_{\zeta}g=0$.

If $2\lambda^2\int_M\left(\Ric(\zeta,\zeta)-\Vert\nabla\zeta\Vert^2\right)\geq n\int_M(\mu-r)^2$, then $\pounds_{\zeta}\pounds_{\zeta}g=0$, hence, $\lambda \pounds_{\zeta}g=(\mu-r)g$.  By taking the trace, we get $0=2\lambda \div(\zeta)=n(\mu-r)$ and we obtain $r=\mu$ (constant) and $\pounds_{\zeta}g=0$.
\end{proof}

In the particular case when the soliton posseses a $\zeta(\Ric)$-vector field \cite{ri}, we prove the following result.

\begin{proposition}
A closed hyperbolic Yamabe soliton $(M^n,g,\zeta,\lambda,\mu)$ with a $\zeta(\Ric)$-vector field satisfying $\nabla \zeta=aQ$, $a\in \mathbb R^*$, and $\trace(\pounds_{\zeta}\Ric)=0$, such that $n+2a\lambda\neq 0$
is a Ricci-flat manifold and $\zeta$ is a parallel vector field. The converse implication holds true.
\end{proposition}
\begin{proof}
We get
$$\div(\zeta)=ar, \ \ \pounds_{\zeta}g=2a\Ric,  \ \ \pounds_{\zeta}\pounds_{\zeta}g=2a\pounds_{\zeta}\Ric$$
and the soliton equation becomes
$$2a\pounds_{\zeta}\Ric+2a\lambda \Ric=(\mu-r)g.$$
Since $\trace(\pounds_{\zeta}\Ric)=0$, we get $(n+2a\lambda) r=n\mu$. It follows that $r$ is a constant provided that $n+2a\lambda\neq 0$. Since $\trace(\pounds_{\zeta}\pounds_{\zeta}g)=0$, from \eqref{e3}, we obtain
$$\Vert \nabla {\zeta}\Vert^2+\div(\nabla_{\zeta}\zeta)-\Ric(\zeta,\zeta)=0,$$
and from \eqref{e2}, we get
$$\Ric(\zeta,\zeta)=\frac{1}{2}\left(\div(\pounds_{\zeta}g)\right)(\zeta)-\zeta(\div(\zeta))=0.$$
Since
$$\Vert \nabla {\zeta}\Vert^2=a^2\Vert Q\Vert^2, \ \ \div(\nabla_{\zeta}\zeta)=a \div(Q\zeta),$$
we infer
\begin{align}\label{e4}
a^2\Vert Q\Vert^2+a \div(Q\zeta)&=0.
\end{align}
By integrating \eqref{e4}, we get $Q=0$, then $\nabla \zeta=0$, hence we get the conclusion.
\end{proof}

\section{Hyperbolic Yamabe soliton submanifolds}

Let $\bar M$ be a smooth manifold with a Riemannian metric $\bar g$, and let $M$ be an isometrically immersed submanifold of it with the induced metric $g$.
For any vector fields $X,Y$ tangent to $M$, and any normal vector field $V$, the Gauss and Weingarten equations are
$$\bar \nabla_XY=\nabla_XY+h(X,Y), \ \ \bar \nabla_V X=-A_VX+\nabla^{\bot}_VX,$$
where $\bar \nabla$ and $\nabla$ are the Levi-Civita connections of $\bar g$ and $g$, $h$ is the second fundamental form, $A$ is the shape operator, and $\nabla^{\bot}$ is the normal connection.

\medskip

We shall further assume that $M$ is a hyperbolic Yamabe soliton having as potential vector field the tangential component $\zeta^{\top}$ of a concurrent vector field $\zeta$ on the manifold $(\bar M, \bar g)$. Then, $\bar \nabla \zeta=I$, where $I$ is the identity map, and for any vector fields $X,Y$ tangent to $M$, we have \cite{ac4}
\begin{align*}
\nabla_X\zeta^{\top}&=X+A_{\zeta^{\top}}X,\\
(\pounds_{\zeta^{\top}}g)(X,Y)&=2\Big(g(X,Y)+g(A_{\zeta^{\bot}}X,Y)\Big),\\
(\pounds_{\zeta^{\top}}\pounds_{\zeta^{\top}}g)(X,Y)&=2\Big(2g(X,Y)+4g(A_{\zeta^{\bot}}X,Y)+2g(A_{\zeta^{\bot}}^2X,Y)\\
&\hspace{12pt}+g((\nabla_{\zeta^{\top}}A_{\zeta^{\bot}})X,Y)\Big).
\end{align*}

Let us recall \cite{cn, k} that a hypersurface is called a \textit{metallic shaped hypersurface} if its shape operator $A$ satisfies
$$A^2=aA+bI,$$
for some real numbers $a$ and $b$. It is know that, in the case of a hypersurface of a space of constant curvature, if the above relation is satisfied in a point, then the hypersurface is pseudosymmetric (for more details, see \cite{Deszcz, DVY}).

Now we can state the following results.

\begin{theorem}
If $(M,g,\zeta^{\top},\lambda,\mu)$ is a hyperbolic Yamabe soliton hypersurface with parallel shape operator (i.e., $\nabla A_{\zeta^{\bot}}=0$), then it is a metallic shaped hypersurface if and only if its scalar curvature is constant.
\end{theorem}
\begin{proof}
From the soliton equation \eqref{e1} and by using the above relations, we get
$$(2\lambda+4)g(X,Y)+(2\lambda+8)g(A_{\zeta^{\bot}}X,Y)+4g(A_{\zeta^{\bot}}^2X,Y)=(\mu-r)g(X,Y)$$
for any vector fields $X,Y$ tangent to $M$, hence
$$A_{\zeta^{\bot}}^2=-\frac{\lambda+4}{2}A_{\zeta^{\bot}}-\frac{2\lambda-\mu+r+4}{4}I,$$
and we obtain the conclusion.
\end{proof}

Based on the results from \cite{DVY}, we have
\begin{proposition}
If $(M,g,\zeta^{\top},\lambda,\mu)$ is a hyperbolic Yamabe soliton hypersurface isometrically immersed into a space of constant curvature and having parallel shape operator, then it is a pseudosymmetric hypersurface.
\end{proposition}

\begin{proposition}
Let $(M,g,\zeta^{\top},\lambda,\mu)$ be a $\zeta^{\top}$-totally umbilical hyperbolic Yamabe soliton with $A_{\zeta^{\bot}}=fI$, for $f$ a smooth function on $M$. If $f$ is constant on the integral curves of $\zeta^{\top}$, then $f$ is constant on the integral curves of $r$.
\end{proposition}
\begin{proof}
In this case,
$$A_{\zeta^{\bot}}^2X=f^2X, \ \ (\nabla_{\zeta^{\top}}A_{\zeta^{\bot}})X=\zeta^{\top}(f)X$$
for any vector field $X$ tangent to $M$, and we get
$$4f^2+2(\lambda+4)f+2\zeta^{\top}(f)+2\lambda-\mu+r+4=0,$$
which, by differentiating in the direction of $\zeta^{\top}$, implies
$$\zeta^{\top}(r)=-2\zeta^{\top}(\zeta^{\top}(f))-2(4f+\lambda+4)\zeta^{\top}(f)=0,$$
hence we get the conclusion.
\end{proof}

As a consequence for \textit{totally geodesic submanifolds} (i.e., submanifolds with vanishing shape operator), we deduce

\begin{theorem}
A totally geodesic submanifold is a hyperbolic Yamabe soliton with potential vector field $\zeta^{\top}$ if and only if it has constant scalar curvature.
\end{theorem}

\begin{proposition}
If the hyperbolic Yamabe soliton $(M^n,g,\zeta^{\top},\lambda,\mu)$ is a compact minimal submanifold with $\lambda \neq 0$ and $\pounds_{\zeta^{\top}}\pounds_{\zeta^{\top}}g$ is divergence-free, then
$$\int_M\Vert A_{\zeta^{\bot}}\Vert^2=\int_M\left(n(n-1)+\frac{\zeta^{\top}(r)}{2\lambda}\right).$$
\end{proposition}
\begin{proof}
Since $M$ is compact, we have \cite{ya}
$$\int_M\left(\Ric(\zeta^{\top},\zeta^{\top})+\frac{1}{2}\Vert \pounds_{\zeta^{\top}}g\Vert^2-\Vert \nabla\zeta^{\top}\Vert^2-(\div(\zeta^{\top}))^2\right)=0.$$
By direct computations, we get
\begin{align*}
\Vert \pounds_{\zeta^{\top}}g\Vert^2&=4\left(\Vert A_{\zeta^{\bot}}\Vert^2+2\trace(A_{\zeta^{\bot}})+n\right)=4\left(\Vert A_{\zeta^{\bot}}\Vert^2+n\right)\\
\Vert \nabla\zeta^{\top}\Vert^2&=\Vert A_{\zeta^{\bot}}\Vert^2+2\trace(A_{\zeta^{\bot}})+n=\Vert A_{\zeta^{\bot}}\Vert^2+n\\
(\div(\zeta^{\top}))^2&=\left(\trace(A_{\zeta^{\bot}})\right)^2+2n\trace(A_{\zeta^{\bot}})+n^2=n^2
\end{align*}
since $M$ is a minimal submanifold.
Also, from $\div(\pounds_{\zeta^{\top}}\pounds_{\zeta^{\top}}g)=0$ and $\lambda \neq 0$, we have $\div(\pounds_{\zeta^{\top}}g)=-\displaystyle\frac{1}{\lambda}dr$, and from \eqref{e2}, we obtain
$$\Ric(\zeta^{\top},\zeta^{\top})=-\frac{\zeta^{\top}(r)}{2\lambda},$$
hence we get the conclusion.
\end{proof}

\begin{proposition}\label{hj}
If the gradient hyperbolic Yamabe soliton $(M^n,g,\zeta^{\top},\lambda,\mu)$ is a minimal submanifold with $\zeta^{\top}=\nabla f$, $\lambda \neq 0$ and $\pounds_{\zeta^{\top}}\pounds_{\zeta^{\top}}g$ is divergence-free, then
$$\frac{1}{2}\Delta(\Vert \nabla f\Vert^2)=
\Vert A_{\nabla f}\Vert^2+n-\frac{g(\nabla f,\nabla r)}{2\lambda}.$$

Moreover, if $M$ is closed,
then $$\int_M\Vert A_{\nabla f}\Vert^2=\frac{1}{2\lambda}\int_Mg(\nabla f,\nabla r).$$
\end{proposition}
\begin{proof}
Using the Bochner's formula
$$\frac{1}{2}\Delta(\Vert \zeta^{\top}\Vert^2)=\Vert \nabla\zeta^{\top}\Vert^2+\zeta^{\top}(\div(\zeta^{\top}))+\Ric(\zeta^{\top},\zeta^{\top}),$$
and considering the above computations, we get the first assertion.

If $M$ is closed, by integrating this relation and taking into account that $n=\Delta(f)$, we deduce the second relation,
hence we get the conclusion.
\end{proof}

\begin{remark}
Under the hypotheses of Proposition \ref{hj}, if $M$ is closed and $\lambda g(\nabla f,\nabla r)\leq 0$, then $M$ is a totally geodesic submanifold and $\Vert \nabla f\Vert^2$ is a subharmonic function (i.e., $\Delta(\Vert \nabla f\Vert^2)\geq 0$).
\end{remark}

\section{Factor manifolds of warped products as \\ hyperbolic Yamabe solitons}

We shall further determine the condition for a factor manifold of a manifold of warped product-type to be a hyperbolic Yamabe soliton having as a potential vector field the corresponding component of a $2$-Killing vector field on the manifold. We recall \cite{op} that a \textit{$2$-Killing vector field} $\zeta$ on a Riemannian manifold $(M,g)$ is a smooth vector field which satisfies $\pounds_{\zeta}\pounds_{\zeta}g=0$. We mention that characterizations of $2$-Killing vector fields on manifolds of warped product-type have been recently given by the present authors in \cite{aeq, ac0, mult} and for the warped product case, by Shenawy and \"{U}nal in \cite{sh}.

We will briefly recall the definitions of the warped product, doubly warped product, multiply warped product, and multiply twisted product manifolds. Let $(M_i,g_i)$, $i\in \{1,2\}$, be two Riemannian manifolds, let $M:=M_1\times M_2$, let $\pi:M\rightarrow M_1$ be the canonical projection, and let
$f:M_1\rightarrow \mathbb R_+^*$ be a smooth function. In 1969, Bishop and O'Neill \cite{Bi} introduced the notion of \textit{warped product manifold} as being a Riemannian manifold
$$(M,\bar g)=:M_1\times _f M_2 \ \ \textrm{with} \ \ \bar g=\pi_{1}^{*}(g_1)+(\pi_1^*(f))^2 \pi_{2}^{*}(g_2).$$
In 1974, Ehrlich \cite{erlich} defined the concept of \textit{doubly warped product manifold} as being a Riemannian manifold
$$(M,\bar g)=:_{f_2}M_1\times_{f_1}  M_2 \ \ \textrm{with} \ \ \bar g=(\pi_2^*(f_2))^2 \pi_{1}^{*}(g_1)+(\pi_1^*(f_1))^2 \pi_{2}^{*}(g_2),$$ where $f_1:M_1\rightarrow \mathbb R_+^*$ and $f_2:M_2\rightarrow \mathbb R_+^*$ are smooth functions.

This notions have been later extended to a larger number of manifolds.
Let $(M_i,g_i)$, $i\in \{1,\dots,n\}$, be Riemannian manifolds, let $M:=M_1\times \dots \times M_n$, and let $\pi_i:M\rightarrow M_i$, $i\in \{1,\dots,n\}$, be the canonical projections. Then:

(1) $(M,\bar g)=:M_1\!\times_{{f_2}}\!\!M_2\times \cdots\times_{f_n}\!\!M_n$ is called a \textit{multiply warped product manifold} \cite{unal} if
$$\bar g=\pi_1^*(g_1)+\sum_{i=2}^n\left(\pi_{1,i}^*(f_i)\right)^2\pi_i^*(g_i),$$
where $\pi_{1,i}:M\rightarrow M_1\times M_i$ is the canonical projection and $f_i:M_1\rightarrow \mathbb R_+^*$ is a smooth function, $i\in\{2,\dots,n\}$;

(2) $(M,\bar g)=:M_1\!\times_{{f_2}}\!\!M_2\times \cdots\times_{f_n}\!\!M_n$ is called a \textit{multiply twisted product manifold} \cite{wang} if
$$\bar g=\pi_1^*(g_1)+\sum_{i=2}^n\left(\pi_{1,i}^*(f_i)\right)^2\pi_i^*(g_i),$$
where $\pi_{1,i}:M\rightarrow M_1\times M_i$ is the canonical projection and $f_i:M_1\times M_i\rightarrow \mathbb R_+^*$ is a smooth function, $i\in\{2,\dots,n\}$.

In the entire section we shall use the same notation for a function on a factor manifold and for a metric on a factor manifold and their pullbacks
on the product space.

Let $\zeta=\sum_{i=1}^n\zeta_i$, $X=\sum_{i=1}^nX_i$, $Y=\sum_{i=1}^nY_i$ be vector fields tangent to $M$, where $\zeta_i,X_i,Y_i$ are the lifts of some vector fields tangent to $M_i$ (denoted in the same way).

\begin{proposition}\label{p1}
If $\zeta$ is a $2$-Killing vector field on the multiply twisted product manifold $M_1\!\times_{{f_2}}\!\!M_2\times \cdots\times_{f_n}\!\!M_n$, then:

(i) $(M_1,g_1,\zeta_1)$ is a hyperbolic Yamabe soliton if and only if there exist $\lambda,\mu\in \mathbb{R}$ such that
$$\pounds_{\zeta_1}g_1=\frac{\mu-r_1}{\lambda}g_1,$$
provided that $\lambda\neq 0$, where $r_1$ is the scalar curvature of $(M_1,g_1)$;

(ii) $(M_i,g_i,\zeta_i)$, $2\leq i\leq n$, is a hyperbolic Yamabe soliton if and only if there exist $\lambda,\mu\in \mathbb{R}$ such that
$$\pounds_{\zeta_i}g_i=\frac{(\mu-r_i)f_i^2+\zeta_1(\zeta_1(f_i^2))+\zeta_1(\zeta_i(f_i^2))}{\lambda f_i^2-2\left(\zeta_1(f_i^2)+\zeta_i(f_i^2)\right)}g_i,$$
provided that $\lambda f_i^2\neq 2\left(\zeta_1(f_i^2)+\zeta_i(f_i^2)\right)$ in any point, where $r_i$ is the scalar curvature of $(M_i,g_i)$.
\end{proposition}
\begin{proof}
From \eqref{e1} and \cite{aeq} it follows that
\begin{align*}
(\pounds_{\zeta}\pounds_{\zeta}\bar g)(X,Y)&=(\pounds_{\zeta_1}\pounds_{\zeta_1}g_1)(X_1,Y_1)+\sum_{i=2}^nf_i^2(\pounds_{\zeta_i}\pounds_{\zeta_i}g_i)(X_i,Y_i)\\
&\hspace{12pt}+2\sum_{i=2}^n\left(\zeta_1(f_i^2)+\zeta_i(f_i^2)\right)(\pounds_{\zeta_i}g_i)(X_i,Y_i)\\
&\hspace{12pt}+\sum_{i=2}^n\left(\zeta_1(\zeta_1(f_i^2))+\zeta_1(\zeta_i(f_i^2))\right)g_i(X_i,Y_i)
\end{align*}
for any vector fields $X_i,Y_i$ tangent to $M_i$, where $\bar g:=g_1+\sum_{i=2}^nf_i^2g_i$ for $f_i$ positive smooth functions defined on $M_1\times M_i$.
Indeed, since $\zeta$ is a $2$-Killing vector field, we have
$$\left\{
    \begin{array}{ll}
      \pounds_{\zeta_1}\pounds_{\zeta_1}g_1=0 \\
      f_i^2(\pounds_{\zeta_i}\pounds_{\zeta_i}g_i)+2\left(\zeta_1(f_i^2)+\zeta_i(f_i^2)\right)(\pounds_{\zeta_i}g_i)+\left(\zeta_1(\zeta_1(f_i^2))+\zeta_1(\zeta_i(f_i^2))\right)g_i=0
    \end{array}
  \right.
$$
for $2\leq i\leq n$. Then $$\pounds_{\zeta_i}\pounds_{\zeta_i}g_i=(\mu-r_i)g_i-\lambda \pounds_{\zeta_i}g_i, \ \ 1\leq i\leq n$$
if and only if $\pounds_{\zeta_i}g_i$ has the expression given by (i) or (ii), respectively, for $1\leq i\leq n$.
\end{proof}

\begin{proposition}\label{p2}
If $\zeta$ is a $2$-Killing vector field on the multiply warped product manifold $M_1\!\times_{{f_2}}\!\!M_2\times \cdots\times_{f_n}\!\!M_n$, then:

(i) $(M_1,g_1,\zeta_1)$ is a hyperbolic Yamabe soliton if and only if there exist $\lambda,\mu\in \mathbb{R}$ such that
$$\pounds_{\zeta_1}g_1=\frac{\mu-r_1}{\lambda}g_1,$$
provided that $\lambda\neq 0$, where $r_1$ is the scalar curvature of $(M_1,g_1)$;

(ii) $(M_i,g_i,\zeta_i)$, $2\leq i\leq n$, is a hyperbolic Yamabe soliton if and only if there exist $\lambda,\mu\in \mathbb{R}$ such that
$$\pounds_{\zeta_i}g_i=\frac{(\mu-r_i)f_i^2+\zeta_1(\zeta_1(f_i^2))}{\lambda f_i^2-2\zeta_1(f_i^2)}g_i,$$
provided that $\lambda f_i^2\neq 2\zeta_1(f_i^2)$ in any point, where $r_i$ is the scalar curvature of $(M_i,g_i)$.
\end{proposition}
\begin{proof}
From \eqref{e1} and \cite{mult} it follows that
\begin{align*}
(\pounds_{\zeta}\pounds_{\zeta}\bar g)(X,Y)&=(\pounds_{\zeta_1}\pounds_{\zeta_1}g_1)(X_1,Y_1)+\sum_{i=2}^nf_i^2(\pounds_{\zeta_i}\pounds_{\zeta_i}g_i)(X_i,Y_i)\\
&\hspace{12pt}+2\sum_{i=2}^n\zeta_1(f_i^2)(\pounds_{\zeta_i}g_i)(X_i,Y_i)+\sum_{i=2}^n\zeta_1(\zeta_1(f_i^2))g_i(X_i,Y_i)
\end{align*}
for any vector fields $X_i,Y_i$ tangent to $M_i$, where $\bar g:=g_1+\sum_{i=2}^nf_i^2g_i$ for $f_i$ positive smooth functions defined on $M_1$.
Indeed, since $\zeta$ is a $2$-Killing vector field, we have
$$\left\{
    \begin{array}{ll}
      \pounds_{\zeta_1}\pounds_{\zeta_1}g_1=0 \\
      f_i^2(\pounds_{\zeta_i}\pounds_{\zeta_i}g_i)+2\zeta_1(f_i^2)(\pounds_{\zeta_i}g_i)+\zeta_1(\zeta_1(f_i^2))g_i=0
    \end{array}
  \right.
$$
for $2\leq i\leq n$. Then
$$\pounds_{\zeta_i}\pounds_{\zeta_i}g_i=(\mu-r_i)g_i-\lambda \pounds_{\zeta_i}g_i, \ \ 1\leq i\leq n$$
if and only if $\pounds_{\zeta_i}g_i$ has the expression given by (i) or (ii), respectively, for $1\leq i\leq n$.
\end{proof}

\begin{proposition}\label{p3}
If $\zeta$ is a $2$-Killing vector field on the doubly warped product manifold $_{{f_2}}M_1\!\times_{{f_1}}\!\!M_2$, then $(M_1,g_1,\zeta_1)$ is a hyperbolic Yamabe soliton if and only if there exist $\lambda,\mu\in \mathbb{R}$ such that
$$\pounds_{\zeta_1}g_1=\frac{(\mu-r_1)f_2^2+\zeta_2(\zeta_2(f_2^2))}{\lambda f_2^2-2\zeta_2(f_2^2)}g_1,$$
provided that $\lambda f_2^2\neq 2\zeta_2(f_2^2)$ in any point, where $r_1$ is the scalar curvature of $(M_1,g_1)$.

Similarly, for $(M_2,g_2,\zeta_2)$.
\end{proposition}
\begin{proof}
From \eqref{e1} and \cite{ac0} it follows that
\begin{align*}
(\pounds_{\zeta}\pounds_{\zeta}\bar g)(X,Y)&=f_2^2(\pounds_{\zeta_1}\pounds_{\zeta_1}g_1)(X_1,Y_1)+f_1^2(\pounds_{\zeta_2}\pounds_{\zeta_2}g_2)(X_2,Y_2)\\
&\hspace{12pt}+2\zeta_2(f_2^2)(\pounds_{\zeta_1}g_1)(X_1,Y_1)+2\zeta_1(f_1^2)(\pounds_{\zeta_2}g_2)(X_2,Y_2)\\
&\hspace{12pt}+\zeta_2(\zeta_2(f_2^2))g_1(X_1,Y_1)+\zeta_1(\zeta_1(f_1^2))g_2(X_2,Y_2)
\end{align*}
for any vector fields $X_i,Y_i$ tangent to $M_i$, where $\bar g:=f_2^2g_1+f_1^2g_2$ for $f_i$ positive smooth function defined on $M_i$.
Indeed, since $\zeta$ is a $2$-Killing vector field, we have
$$\left\{
    \begin{array}{ll}
      f_2^2(\pounds_{\zeta_1}\pounds_{\zeta_1}g_1)+2\zeta_2(f_2^2)(\pounds_{\zeta_1}g_1)+\zeta_2(\zeta_2(f_2^2))g_1=0 \\
      f_1^2(\pounds_{\zeta_2}\pounds_{\zeta_2}g_2)+2\zeta_1(f_1^2)(\pounds_{\zeta_2}g_2)+\zeta_1(\zeta_1(f_1^2))g_2=0
    \end{array}
  \right..
$$
Then $$\pounds_{\zeta_1}\pounds_{\zeta_1}g_1=(\mu-r_1)g_1-\lambda \pounds_{\zeta_1}g_1$$
if and only if $\pounds_{\zeta_1}g_1$ has the above given expression.
\end{proof}

\begin{proposition}\label{p4}
If $\zeta$ is a $2$-Killing vector field on the warped product manifold $M_1\!\times_{{f}}\!\!M_2$, then:

(i) $(M_1,g_1,\zeta_1)$ is a hyperbolic Yamabe soliton if and only if there exist $\lambda,\mu\in \mathbb{R}$ such that
$$\pounds_{\zeta_1}g_1=\frac{\mu-r_1}{\lambda}g_1,$$
provided that $\lambda\neq 0$, where $r_1$ is the scalar curvature of $(M_1,g_1)$;

(ii) $(M_2,g_2,\zeta_2)$ is a hyperbolic Yamabe soliton if and only if there exist $\lambda,\mu\in \mathbb{R}$ such that
$$\pounds_{\zeta_2}g_2=\frac{(\mu-r_2)f^2+\zeta_1(\zeta_1(f^2))}{\lambda f^2-2\zeta_1(f^2)}g_2,$$
provided that $\lambda f^2\neq 2\zeta_1(f^2)$ in any point, where $r_2$ is the scalar curvature of $(M_2,g_2)$.
\end{proposition}
\begin{proof}
This is a particular case of Proposition \ref{p3} for $f_1=f$ and $f_2=1$.
\end{proof}

We can now provide the following examples.

\begin{example}
Let $M\!\times_{{f}}\!I$ be a warped product spacetime with a $2$-Killing vector field $\zeta$, where $I$ is an open real interval endowed with the metric $-dt^2$, and $(M,g)$ is a Riemannian manifold of constant scalar curvature $r$. Then $(M,g,\zeta_1,\lambda,r)$ with $\lambda\neq 0$ is a hyperbolic Yamabe soliton, by means of Proposition \ref{p4} (i).
\end{example}

\begin{example}
Let $I\!\times_{{f}}\!M$ be a warped product spacetime with a $2$-Killing vector field $\zeta$, where $I$ is an open real interval endowed with the metric $-dt^2$, and $(M,g)$ is a Riemannian manifold of constant scalar curvature $r$. If the warping function $f:I\rightarrow \mathbb R^*_+$ satisfies
$$f^2(t)=c_1e^{\sqrt{|\mu-r|}t}+c_2e^{-\sqrt{|\mu-r|}t}, \ \ c_1,c_2\in \mathbb{R},$$
then $(M,g,\zeta_2)$ is a hyperbolic Yamabe soliton, by means of Proposition \ref{p4} (ii).
\end{example}

Since under a certain assumption (when the denominator is nowhere zero), which we shall call condition $(*)$, in all the above cases from Propositions \ref{p1}, \ref{p2}, \ref{p3}, \ref{p4} the first Lie derivative of the metric is proportional to the Riemannian metric, based on Theorem 2.1 from \cite{zhang}, we obtain the following classification result.

\begin{theorem}
If any of the factor manifolds $(M_i,g_i,\zeta_i)$ of a multiply twisted, multiply warped, doubly warped, or warped product manifold endowed with a $2$-Killing vector field $\zeta$ is a complete hyperbolic Yamabe soliton satisfying condition $(*)$ and having as potential vector field a vector field of gradient type, then the product manifold is isometric to a complete warped product manifold of one of the following forms:

(i) $\mathbb R\times_{f} N$ with $g=dt^2+f^2g_N$, for $(N,g_N)$ an $(n_i-1)$-dimensional Riemannian manifold;

(ii) $\mathbb R\times_{f} \mathbb R^{n_i-1}$ with $g=dt^2+f^2dt_{n_i-1}$, $t\geq 0$;

(iii) $\mathbb S\times_{f} \mathbb S^{n_i-1}$ with $g=dt^2+f^2dt_{n_i-1}$, $t\in [a,b]$.
\end{theorem}

We shall further determine the conditions for the factor manifolds to be hyperbolic Yamabe solitons when the manifold is a hyperbolic Yamabe soliton.

\begin{proposition}
If the multiply twisted product manifold $(M_1\!\times_{{f_2}}\!\!M_2\times \cdots\times_{f_n}\!\!M_n,\zeta)$ is a hyperbolic Yamabe soliton, then:

(i) $(M_1,g_1,\zeta_1)$ is a hyperbolic Yamabe soliton if and only if $\bar r-r_1$ is constant;

(ii) $(M_i,g_i,\zeta_i)$, $2\leq i\leq n$, is a hyperbolic Yamabe soliton if and only if
$$\left\{
  \begin{array}{ll}
    \frac{\displaystyle \zeta_1(f_i^2)+\zeta_i(f_i^2)}{\displaystyle f_i^2} \ \ \textrm{is constant}\\
    \bar r-r_i+\frac{\displaystyle \zeta_1(\zeta_1(f_i^2))+\zeta_1(\zeta_i(f_i^2))}{\displaystyle f_i^2}  \ \ \textrm{is constant}
  \end{array}
\right.,$$
where $\bar r$ is the scalar curvature of the multiply twisted product manifold and $r_i$ is the scalar curvature of $(M_i,g_i)$.
\end{proposition}
\begin{proof}
It follows from \eqref{e1}, from the expression of the second Lie derivative of the metric given in Proposition \ref{p1}, and \cite{aeq}
\begin{align*}
(\pounds_{\zeta}\bar g)(X,Y)&=(\pounds_{\zeta_1}g_1)(X_1,Y_1)\\
&\hspace{12pt}+\sum_{i=2}^nf_i^2(\pounds_{\zeta_i}g_i)(X_i,Y_i)+\sum_{i=2}^n\left(\zeta_1(f_i^2)+\zeta_i(f_i^2)\right)g_i(X_i,Y_i)
\end{align*}
for any vector fields $X_i,Y_i$ tangent to $M_i$, where $\bar g:=g_1+\sum_{i=2}^nf_i^2g_i$ for $f_i$ positive smooth functions defined on $M_1\times M_i$.
\end{proof}

\begin{proposition}
If the multiply warped product manifold $(M_1\!\times_{{f_2}}\!\!M_2\times \cdots\times_{f_n}\!\!M_n,\zeta)$ is a hyperbolic Yamabe soliton, then:

(i) $(M_1,g_1,\zeta_1)$ is a hyperbolic Yamabe soliton if and only if $\bar r-r_1$ is constant;

(ii) $(M_i,g_i,\zeta_i)$, $2\leq i\leq n$, is a hyperbolic Yamabe soliton if and only if
$$\left\{
  \begin{array}{ll}
    \frac{\displaystyle \zeta_1(f_i^2)}{\displaystyle f_i^2} \ \ \textrm{is constant}\\
    \bar r-r_i+\frac{\displaystyle \zeta_1(\zeta_1(f_i^2))}{\displaystyle f_i^2}  \ \ \textrm{is constant}
  \end{array}
\right.,$$
where $\bar r$ is the scalar curvature of the multiply warped product manifold and $r_i$ is the scalar curvature of $(M_i,g_i)$.
\end{proposition}
\begin{proof}
It follows from \eqref{e1}, from the expression of the second Lie derivative of the metric given in Proposition \ref{p2}, and \cite{mult}
\begin{align*}
(\pounds_{\zeta}\bar g)(X,Y)&=(\pounds_{\zeta_1}g_1)(X_1,Y_1)\\
&\hspace{12pt}+\sum_{i=2}^nf_i^2(\pounds_{\zeta_i}g_i)(X_i,Y_i)+\sum_{i=2}^n\zeta_1(f_i^2)g_i(X_i,Y_i)
\end{align*}
for any vector fields $X_i,Y_i$ tangent to $M_i$, where $\bar g:=g_1+\sum_{i=2}^nf_i^2g_i$ for $f_i$ positive smooth functions defined on $M_1$.
\end{proof}

\begin{proposition}\label{p7}
If the doubly warped product manifold $(_{{f_2}}M_1\!\times_{{f_1}}\!\!M_2,\zeta)$ is a hyperbolic Yamabe soliton, then
$(M_1,g_1,\zeta_1)$ is a hyperbolic Yamabe soliton if and only if
$$\left\{
  \begin{array}{ll}
    \frac{\displaystyle \zeta_2(f_2^2)}{\displaystyle f_2^2} \ \ \textrm{is constant}\\
    \bar r-r_1+\frac{\displaystyle \zeta_2(\zeta_2(f_2^2))}{\displaystyle f_2^2}  \ \ \textrm{is constant}
  \end{array}
\right.,$$
where $\bar r$ is the scalar curvature of the doubly warped product manifold and $r_1$ is the scalar curvature of $(M_1,g_1)$.

Similarly, for $(M_2,g_2,\zeta_2)$.
\end{proposition}
\begin{proof}
It follows from \eqref{e1}, from the expression of the second Lie derivative of the metric given in Proposition \ref{p3}, and \cite{ac0}
\begin{align*}
(\pounds_{\zeta}\bar g)(X,Y)&=f_2^2(\pounds_{\zeta_1}g_1)(X_1,Y_1)+f_1^2(\pounds_{\zeta_2}g_2)(X_2,Y_2)\\
&\hspace{12pt}+\zeta_2(f_2^2)g_1(X_1,Y_1)+\zeta_1(f_1^2)g_2(X_2,Y_2)
\end{align*}
for any vector fields $X_i,Y_i$ tangent to $M_i$, where $\bar g:=f_2^2g_1+f_1^2g_2$ for $f_i$ positive smooth functions defined on $M_i$.
\end{proof}

\begin{proposition}\label{p5}
If the warped product manifold $(M_1\!\times_{{f}}\!M_2,\zeta)$ is a hyperbolic Yamabe soliton, then:

(i) $(M_1,g_1,\zeta_1)$ is a hyperbolic Yamabe soliton if and only if $\bar r-r_1$ is constant;

(ii) $(M_2,g_2,\zeta_2)$ is a hyperbolic Yamabe soliton if and only if
$$\left\{
  \begin{array}{ll}
    \zeta_1(\ln f) \ \ \textrm{is constant}\\
    \bar r-r_2+2\Big[\left(\frac{\displaystyle \zeta_1(f)}{\displaystyle f}\right)^2+
\frac{\displaystyle \zeta_1(\zeta_1(f))}{\displaystyle f}\Big]  \ \ \textrm{is constant}
  \end{array}
\right.,$$
where $\bar r$ is the scalar curvature of the warped product manifold, $r_1$ is the scalar curvature of $(M_1,g_1)$, and $r_2$ is the scalar curvature of $(M_2,g_2)$.
\end{proposition}
\begin{proof}
This is a particular case of the Proposition \ref{p7} for $f_1=f$ and $f_2=1$. In this case, at (ii),
$$\frac{\displaystyle \zeta_1(f^2)}{\displaystyle f^2}=2\zeta_1(\ln f)$$
and
$$\bar r-r_2+\frac{\displaystyle \zeta_1(\zeta_1(f^2))}{\displaystyle f^2}=\bar r-r_2+2\Big[\left(\frac{\displaystyle \zeta_1(f)}{\displaystyle f}\right)^2+
\frac{\displaystyle \zeta_1(\zeta_1(f))}{\displaystyle f}\Big]$$
must be constant.
\end{proof}

\begin{remark}\label{r}
We know from \cite{Bi} that the scalar curvature of a warped product manifold $M_1^{n_1}\!\times_{{f}}\!M_2^{n_2}$ is given by
$$\bar r=r_1+\frac{r_2}{f^2}-2n_2\frac{\Delta(f)}{f}-n_2(n_2-1)\frac{\Vert \nabla f\Vert^2}{f^2}.$$
\end{remark}

We shall provide some examples for the warped product manifolds $I\!\times_{{f}}\!\mathbb R^3$ and $\mathbb R^3\!\times_{{f}}\!I$, where $I$ an open real interval endowed with the metric $-dt^2$ and $\mathbb R^3$ with the canonical metric $g_0=dx_1^2+dx_2^2+dx_3^2$.

Let us firstly remark that the scalar curvature of the warped product manifold $I\!\times_{{f}}\!\mathbb R^3$ is equal to
$$\bar r=-6\left(\frac{\Delta(f)}{f}+\frac{\Vert \nabla f\Vert^2}{f^2}\right),$$
and the scalar curvature of the warped product manifold $\mathbb R^3\!\times_{{f}}\!I$ is equal to
$$\bar r=-2\frac{\Delta(f)}{f}.$$

\begin{example}
If the warped product Robertson--Walker spacetime
\linebreak
$(I\!\times_{{f}}\!\mathbb R^3,\zeta)$ is a hyperbolic Yamabe soliton,
then $(I,-dt^2,\zeta_1)$ is a hyperbolic Yamabe soliton, by means of Proposition \ref{p5} (i) and Remark \ref{r} if and only if
$$\frac{\Delta(f)}{f}+\frac{\Vert \nabla f\Vert^2}{f^2} \ \ \textrm{is constant}.$$
This condition is satisfied by $f:I\rightarrow \mathbb R^*_+$, $f(t)=e^{c_1t+c_2}, \ \ c_1,c_2\in \mathbb{R}$, since
$$\frac{\vert \Delta(f)\vert}{f}=c_1^2=\frac{\vert\Vert \nabla f\Vert^2\vert}{f^2}.$$
\end{example}

\begin{example}
If the warped product Robertson--Walker spacetime
\linebreak
$(I\!\times_{{f}}\!\mathbb R^3,\zeta)$ is a hyperbolic Yamabe soliton,
$$\zeta=k\frac{d}{dt}+\zeta_2, \ \ k\in \mathbb{R},$$ and
$$f:I\rightarrow \mathbb R^*_+, \ \ f(t)=e^{c_1t+c_2}, \ \ c_1,c_2\in \mathbb{R},$$
then $(\mathbb R^3,g_0,\zeta_2)$ is a hyperbolic Yamabe soliton, by means of Proposition \ref{p5} (ii) and Remark \ref{r} , since
$$\zeta_1(\ln f)=\frac{\zeta_1(f)}{f}=kc_1, \ \ \frac{\zeta_1(\zeta_1(f))}{f}=(kc_1)^2, \ \ \frac{\vert\Delta(f)\vert}{f}=c_1^2=\frac{\vert\Vert \nabla f\Vert^2\vert}{f^2}.$$
\end{example}

\begin{example}
If the warped product manifold $(\mathbb R^3\!\times_{{f}}\!I,\zeta)$ is a hyperbolic Yamabe soliton,
then $(\mathbb R^3,g_0,\zeta_1)$ is a hyperbolic Yamabe soliton, by means of Proposition \ref{p5} (i) and Remark \ref{r} if and only if
$$\frac{\Delta(f)}{f} \ \ \textrm{is constant}.$$
This condition is satisfied by $f:\mathbb R^3\rightarrow \mathbb R^*_+$, $f(x_1,x_2,x_3)=e^{x_1+x_2+x_3}$, since
$$\frac{\Delta(f)}{f}=3.$$
\end{example}

\begin{example}
If the warped product manifold $(\mathbb R^3\!\times_{{f}}\!I,\zeta)$ is a hyperbolic Yamabe soliton,
$$\zeta=k\frac{\partial }{\partial x_1}+\zeta_2, \ \ k\in \mathbb{R},$$ and
$$f:\mathbb R^3\rightarrow \mathbb R^*_+, \ \ f(x_1,x_2,x_3)=e^{x_1+x_2+x_3},$$
then $(I,-dt^2,\zeta_2)$ is a hyperbolic Yamabe soliton, by means of Proposition \ref{p5} (ii) and Remark \ref{r}, since
$$\zeta_1(\ln f)=\frac{\zeta_1(f)}{f}=kc_1, \ \ \frac{\zeta_1(\zeta_1(f))}{f}=(kc_1)^2, \ \ \frac{\Delta(f)}{f}=3.$$
\end{example}

\begin{example}
If the warped product manifold $(\mathbb R^3\!\times_{{f}}\!I,\zeta)$ is a hyperbolic Yamabe soliton,
$$\zeta=\sum_{i=1}^3k_i\frac{\partial }{\partial x_i}+\zeta_2, \ \ k_i\in \mathbb{R}, \ i=\overline{1,3},$$ and
$$f:\mathbb R^3\rightarrow \mathbb R^*_+, \ \ f(x_1,x_2,x_3)=e^{x_1},$$
then $(I,-dt^2,\zeta_2)$ is a hyperbolic Yamabe soliton, by means of Proposition \ref{p5} (ii) and Remark \ref{r}, since
$$\zeta_1(\ln f)=\frac{\zeta_1(f)}{f}=\sum_{i=1}^3k_i, \ \ \frac{\zeta_1(\zeta_1(f))}{f}=\left(\sum_{i=1}^3k_i\right)^2, \ \ \frac{\Delta(f)}{f}=1.$$
\end{example}

\vspace{2mm} \noindent \footnotesize
\begin{minipage}[b]{10cm}
Adara M. Blaga \\
Department of Mathematics \\
West University of Timi\c{s}oara, Romania \\
Email: adarablaga@yahoo.com
\end{minipage}

\vspace{4mm} \noindent \footnotesize
\begin{minipage}[b]{10cm}
Cihan \"{O}zg\"{u}r\\
Department of Mathematics \\
Izmir Democracy University, T\" urkiye \\
Email: cihan.ozgur@idu.edu.tr
\end{minipage}

\end{document}